\documentclass[12pt]{article}
\usepackage[utf8]{inputenc}
\usepackage[T1]{fontenc}
\usepackage[english]{babel}
\usepackage{amsmath,amssymb,amsthm,bbm}
\usepackage[colorlinks=true]{hyperref}
\usepackage[margin=2cm]{geometry}

\newcommand{\R}{\mathbb R}
\newcommand{\E}{\mathbb E}
\newcommand{\prob}{\mathbb P}
\newcommand{\e}{\mathrm e}
\renewcommand{\phi}{\varphi}
\renewcommand{\epsilon}{\varepsilon}
\renewcommand{\tilde}{\widetilde}

\newtheorem{thm}{Theorem}
\newtheorem{lem}[thm]{Lemma}

\theoremstyle{remark}
\newtheorem*{rem}{Remark}

\theoremstyle{definition}

\title{The Langevin Monte Carlo algorithm 
in the non-smooth log-concave case}
\author{Joseph Lehec}

\begin{document}
\maketitle

\begin{abstract}
We prove non-asymptotic polynomial bounds on the 
convergence of the Langevin Monte Carlo algorithm
in the case where the potential is a convex function 
which is globally Lipschitz on its domain, typically 
the maximum of a finite number of affine functions on
an arbitrary convex set. In particular the 
potential is not assumed to be gradient Lipschitz,
in contrast with most existing works 
on the topic.
\textit{MS Classification:} 62D05 (68W20, 65C05, 52A23) 
\end{abstract}

\section{Introduction}

\paragraph{Setting.}
Sampling from a high-dimensional log-concave probability
measure is a problem dating back to the early nineties 
and the seminal work of Dyer, Frieze and Kannan~\cite{DFK}
and which has many applications to various fields 
such as Bayesian statistics, convex optimization and statistical inference. 
This problem is always addressed via Markov Chain 
Monte Carlo methods, but there is a large variety of those: 
Metropolis-Hastings type random walks (ball 
walk), Glauber like dynamics (hit and run) or Hamiltonian Monte Carlo.  
In this article, we will consider the so-called Langevin algorithm, 
which is defined as follows. 
Given a probability measure $\mu$ on $\R^n$ we let 
$\phi$ be its potential, namely $\mu$ has density 
$\e^{-\phi}$ with respect to the Lebesgue measure. The 
Langevin diffusion associated to $\mu$ is the solution $(X_t)$ 
of the following stochastic differential equation 
\begin{equation}\label{eq:langevin-diff}
dX_t = d B_t - \frac 12 \nabla \phi ( X_t) \, dt  ,
\end{equation}
where $(B_t)$ is a standard $n$-dimensional Brownian motion. 
The Langevin algorithm is the Euler scheme associated 
to this diffusion: Given a time step parameter $\eta$ 
we let $(\xi_k)_{k\geq 1}$ be a sequence of i.i.d. 
centered Gaussian vectors with covariance $\eta \, \text{Id}$ and set 
\begin{equation}\label{eq:langevinalg}
x_{k+1} = x_k + \xi_{k+1} - \frac \eta2 \nabla \phi ( x_k ) . 
\end{equation}
We shall focus on the \emph{log-concave} case, namely the case 
where the potential $\phi$ is convex. One originality 
of this work is that we will consider the \emph{constrained}
case, allowing the measure $\mu$ 
to be supported on a set $K$ different from $\R^n$. 
In other words the potential $\phi$ is allowed to take the
value $+\infty$ outside some set $K$. Notice that 
the log-concavity assumption implies that $K$ is convex. 
In the constrained case the Langevin 
diffusion~\eqref{eq:langevin-diff} becomes 
\begin{equation}
\label{eq:tanaka-diff}
d X_t = d B_t - \frac 12 \nabla \phi ( X_t ) \, dt - d \Phi_t ,
\end{equation}
where $(\Phi_t)$ is a process that repels $(X_t)$ inward when it 
reaches the boundary of $K$, see the next section for a 
precise definition. The discretization then becomes 
\begin{equation}\label{eq:PLMC}
x_{k+1} = \mathcal P \left( x_k + \xi_{k+1} 
- \frac \eta 2 \nabla \phi (x_k) \right) ,
\end{equation}
where $\mathcal P$ is the projection on $K$: 
For $x\in \R^n$ the point $\mathcal P_K (x)$ is the 
closest point to $x$ in $K$. This is the algorithm 
we will study throughout the article. It 
was first introduced in our joint paper 
with Bubeck and Eldan~\cite{BEL} and to the best of our knowledge 
it has not been investigated since. 

The second originality of this work is that we 
do not assume the potential $\phi$ to be smooth. 
More precisely we will assume 
that the gradient of $\phi$ (or rather its subdifferential)
is uniformly bounded on $K$, but we do not assume it to be Lipschitz
or even continuous. Let us point out that this is by no means an exotic situation, 
the reader could think for instance of $\phi$ 
being the maximum of a finite number of affine functions on $K$. 
We do not make any assumption whatsoever on the convex set $K$. 

The drawback of this very generic situation and of our approach
is that we are only able to get convergence estimates in Wasserstein distance. 
Recall that the Wasserstein distance $W_2$ between 
two probability measures $\mu$ and $\nu$ is defined as 
\[
W_2^2 ( \mu , \nu ) = \inf_{X \sim \mu , Y\sim \nu} \left\{ \E [ \vert X - Y \vert^2 ]\right\}. 
\] 
By a slight abuse of notation, if $X,Y$ are random vectors we also write $W_2 (X,Y)$ for the 
Wasserstein distance between the law of $X$ and that of $Y$.
\medskip

\paragraph{Main results.}
Our main result is the following bound between the 
Langevin algorithm~\eqref{eq:PLMC} after $k$ steps
and its corresponding point $X_{k\eta}$ in the true Langevin 
diffusion~\eqref{eq:tanaka-diff}.
\begin{thm}\label{thm:main}
Assume that $\mu$ is log-concave, with globally Lipschitz potential $\phi$ 
on its support $K$ and let $L$ be the Lipschitz constant.
Assume that the time step $\eta$ satisfies $\eta < n L^{-2}$ 
and suppose that the Langevin algorithm and 
diffusion are initiated at the same point
$x_0$. Then for every integer $k$ we have  
\begin{equation}\label{eq:main1}
\frac 1n  W_2^2 ( X_{k\eta} , x_k ) 
\leq A\,  k \, \eta^{3/2} 
\end{equation}
where 
\begin{equation}\label{eq:defA}
A = (2\e^{1/2}+1) \frac {(1+\sigma_0)  (n+2\log k)^{1/2}}{r_0} + \frac 76\, \frac{L}{n^{1/2}} , 
\end{equation}
and 
\begin{equation}\label{eq:sigmar}
r_0= d(x_0,\partial K) \quad \text{and} \quad 
\sigma_0 = \frac 1n \left( \phi (x_0) - \inf_K \{ \phi (y) \} \right) .  
\end{equation}
\end{thm}
\begin{rem}
The transport cost $W_2^2$ behaves additively 
when taking tensor products, so the Wasserstein distance between any 
two probability measures on $\R^n$ is typically of order $\sqrt n$.
Therefore $\frac 1n W_2^2$ is of order $1$, which explains why we wrote the 
theorem this way. The reader should thus have in mind that the theorem provides some
non trivial information as soon as the right-hand side of~\eqref{eq:main1} 
is smaller than some small constant $\epsilon$. 
\end{rem}
The result depends on the starting point via the parameters 
$r_0$ and $\sigma_0$. In order to get a meaningful bound the algorithm 
should not be initiated too close to the boundary of $K$ or at a point 
where the potential $\phi$ is
too large. Let us also point out that the theorem is also valid when 
there is no support constraint, namely when $K=\R^n$. One just replaces 
$r_0$ by $+\infty$, so that $A = O ( L n^{-1/2} )$ in this case. 
Let us comment also on the parameter $\sigma_0$. Obviously $\sigma_0 =0$ 
if the potential is minimal at $x_0$.
Fradelizi's theorem~\cite[Theorem 4]{frad} asserts that if $\mu$ is log-concave on $\R^n$ 
with density $f$ and has its barycenter at $x_0$ then 
\[
\sup_{x\in \R^n} \{ f(x) \} \leq e^n f(x_0) . 
\]
In terms of the parameter $\sigma_0$ this means that if $\mu$ has its 
barycenter at $x_0$ then $\sigma_0 \leq 1$. Since $\phi$ is assumed to 
be Lipschitz with constant $L$, if $x_0$ is at $O(nL^{-1})$ distance from the barycenter then 
again $\sigma_0$ is order $1$. In general we shall think of $\sigma_0$ 
as a parameter of order $1$. Also we are never going to take more than 
$poly(n)$ steps so $\log k$ will always be negligible 
compared to $n$. Under the previous assumptions the parameter $A$ 
thus satisfies
\[
A = O \left(  \max \left( \frac{ n^{1/2} } {r_0} ; \frac{ L }{ n^{1/2} } \right) \right) . 
\]

In order to estimate the complexity of the Langevin algorithm, 
we need to combine the previous theorem with some estimate 
on the speed of convergence of the Langevin diffusion $(X_t)$ 
towards its equilibrium measure $\mu$. For this we shall 
use two functional inequalites, the Poincar\'e inequality 
and the logarithmic Sobolev inequality. 
Recall that the measure $\mu$ is said to satisfy the logarithmic Sobolev inequality
if for every probability measure $\nu$ on $\R^n$ we have 
\begin{equation}\label{eq:logsob}
D ( \nu \mid \mu ) \leq \frac {C_{LS}} 2 \, I ( \nu \mid \mu ) 
\end{equation}
where $D ( \nu \mid \mu )$ and $I(\nu\mid \mu)$ denote respectively
the relative entropy and the relative Fisher information of 
$\nu$ with respect to $\mu$: 
\[
D ( \nu \mid \mu ) = \int_{\R^n} \log \left( \frac{d\nu}{d\mu} \right) \, d\nu 
\quad \text{and} \quad 
I ( \nu \mid \mu ) = \int_{\R^n} \left\vert
\nabla \log \left( \frac{d\nu}{d\mu} \right) \right\vert^2  \, d\nu . 
\]
The smallest constant $C_{LS}$ for which~\eqref{eq:logsob} holds true 
is called the log-Sobolev constant of $\mu$. 
The factor $\frac12$ is just a matter of convention, with this 
normalization the log-Sobolev constant of the
standard Gaussian measure is $1$, in any dimension. It is well-known that the 
log-Sobolev inequality is stronger than the Poincar\'e inequality. 
More precisely, letting $C_P$ be the best constant 
in the Poincar\'e inequality: 
\[
\mathrm{var}_\mu (f) \leq C_P \int_{\R^n} \vert \nabla f \vert^2 \, d\mu ,  
\]
we have $C_P \leq C_{LS}$. 
\begin{thm}\label{thm:logsob}
Again assume that $\mu$ is log-concave with 
globally Lipschitz potential on its support, with Lipschitz 
constant $L$. Let $x_0$ be a point in the support of $\mu$ 
and recall the definition~\eqref{eq:sigmar} of $\sigma_0$ and $r_0$. 
Assume in addition that the measure $\mu$ satisfies the log-Sobolev
inequality with constant $C_{LS}$. 
Then after $k$ steps of the Langevin algorithm 
started at $x_0$ with time step parameter $\eta<nL^{-2}$ we have 
\[
\frac 1n W_2^2 ( x_k , \mu ) 
\leq 2 B \, \e^{-k\eta / 2C_{LS}} + 2 A \, k\eta^{3/2} 
\]
where again $A$ is given by~\eqref{eq:defA} and 
\[
B =  4 C_{LS} \left( 1 + \log \left( \frac{ \max ( C_{LS} ,1 )\, n }{ \min ( r_0 , 1 ) } \right)   
+ \sigma_0 + \frac L n \right) .
\]
\end{thm}
\begin{rem}
Note that we initiate the Langevin algorithm at a Dirac 
point mass, we do not need any \emph{warm start} hypothesis. 
The starting point only plays a role through the parameters $r_0$ 
and $\sigma_0$. 
\end{rem}
Let us describe what the theorem gives in terms of 
the complexity of the Langevin algorithm. Say we want 
$\frac 1n W_2^2 ( x_k ,\mu ) \leq \epsilon$ for some 
small $\epsilon$. The first term of the right-hand side 
is a little bit intricate, so let us assume that all parameters 
of the problem are at most polynomial in the dimension. Then 
that term is just $poly (n) \exp (- k \eta / 2 C_{LS} )$, which 
is negligible as soon as $k\eta = \Omega ( C_{LS} \log n )$.   
Let us also assume that $\sigma_x = O(1)$ (see the discussion above). 
Then the theorem shows that choosing
\[
\eta = \Theta^* \left( \frac {\epsilon^2}{C_{LS}^2} \min \left( \frac {r_0^2} n , 
\frac n{L^2} \right)\right)
\] 
and running the algorithm for   
\[
k  = \Theta^* \left( \frac{ C_{LS}^3 }{ \epsilon^2 } \max \left( \frac n {r_0^2} 
; \frac {L^2} n \right) \right)   
\]
steps produces a point $x_k$ satisfying $\frac 1n W_2^2 ( x_k , \mu ) \leq \epsilon$. 
The notation $\Theta^*$ hides universal constants as well
as possible $polylog(n)$ dependencies, this is a common practice in this field. 

Note in particular that if we treat all parameters other than the dimension as constants, 
then we already get a non trivial information after 
a number of steps of the algorithm which is nearly linear in the dimension.  

Of course not every log-concave measure satisfies log-Sobolev, simply 
because log-Sobolev implies sub-Gaussian tails. 
However there are a number of interesting cases in which the log-Sobolev 
inequality is known to hold true, which we list below. 
\begin{enumerate} 
\item If the potential $\phi$ is $\alpha$-uniformly convex for 
some $\alpha>0$, in the sense that $x\mapsto \phi(x) - \frac \alpha 2 \vert x\vert^2$ 
is convex, then $\mu$ satisfies $\log$-Sobolev
with constant $1/\alpha$. This is the celebrated Bakry-\'Emery criterion, see~\cite{BaGL}.
See also~\cite{BL} for an alternate proof based on the Pr\'ekopa-Leindler inequality.
\item If $\mu$ is log-concave and is supported on a 
ball of radius $R$, then $\mu$ satisfies log-Sobolev with constant $R^2$, up to a universal 
constant. This follows trivially from E. Milman's result that, 
within the class of log-concave measures, 
Gaussian concentration and the log-Sobolev inequality are equivalent, 
see \cite[Theorem 1.2.]{milman-duke}, or \cite[Theorem 2]{ledoux-semigroup}.
\item If $\mu$ is log-concave, supported on a ball of radius $R$ and isotropic,
in the sense that its covariance matrix is the identity matrix, 
then Lee and Vempala~\cite{LV} have shown
that $\mu$ satisfies log-Sobolev with constant $R$, up to a universal 
factor.
Note that the isotropy condition implies that $R\geq \sqrt n$, so this improves
greatly upon the previous result in the isotropic case. 
\end{enumerate}
In the first case, notice that since the potential is at the same time 
globally Lipschitz and uniformly convex, the support of $\mu$ must be bounded. 
Actually if the potential is globally Lipschitz it cannot grow fast enough at infinity
to insure log-Sobolev. So if we insist on assuming that the potential is Lipschitz 
and on using log-Sobolev then we have to assume that the support is bounded. 

One way around this issue is to use the Poincar\'e inequality 
rather than log-Sobolev. Indeed every log-concave measure 
satisfies the Poincar\'e inequality. 
Kannan, Lovasz and Simonovits~\cite{KLS} proved that the 
Poincar\'e constant of an isotropic log-concave measure 
on $\R^n$ is $O(n)$ and conjectured that it should actually 
be bounded. This conjecture, which was the major open problem in the field 
of asymptotic convex geometry, was recently nearly solved by Yuansi Chen~\cite{chen}, 
who proved an $n^{o(1)}$ bound for the Poincar\'e constant of an isotropic 
log-concave vector in dimension $n$. The result of Chen relies 
on a technique invented by Eldan~\cite{eldan} which was also used by
Lee an Vempala~\cite{LV} to prove a $O ( \sqrt n )$ bound for the KLS constant, 
as well as the aforementioned log-Sobolev result. 
Recall that if $\nu$ is a probability measure, absolutely continuous with 
respect to $\mu$, the chi-square divergence of $\nu$ with respect to $\mu$ 
is defined as 
\[
\chi^2 ( \nu\mid \mu ) = \int_{\R^n} \left( \frac {d\nu}{d\mu} -1 \right)^2 \, d\mu . 
\]
Our next theorem then states as follows. 
\begin{thm}\label{thm:poincare}
Assume that $\mu$ is a log-concave probability measure with globally Lipschitz potential 
$\phi$ on its domain, with constant $L$. Then after $k$ steps of the Langevin 
algorithm initiated at a random point $x_0$ taking values in the domain, 
and with time step parameter $\eta$ satisfying $\eta \leq n L^{-2}$, we have
\[
\frac 1n W_2^2 (x_k,\mu) 
\leq \frac 4n \, C_P \chi^2 ( x_0 \mid \mu ) \e^{ - k \eta / C_P } +  2 \, A  k \eta^{3/2} .    
\]
where $C_P$ is the Poincar\'e constant of $\mu$ and where  
\[
A = (2\e^{1/2}+1) (n+2\log k)^{1/2} \, \E \left[ \frac{ 1+\sigma_0 }{ r_0 } \right] 
+ \frac 76 \, \frac{L}{n^{1/2}} .
\]
Note that $r_0$ and $\sigma_0$ are random here.
\end{thm}
So the price we have to pay for using Poincar\'e rather than 
log-Sobolev is a warm start hypothesis: the 
algorithm must be initiated at a random point $x_0$ whose 
chi-square divergence to $\mu$ is finite. In the unconstrained case, 
namely when $\mu$ is supported on the whole $\R^n$, 
a natural choice for a warm start is an appropriate Gaussian
measure. One can indeed get the following estimate. 
\begin{lem}\label{lem:chi2}
Suppose $\mu$ is log-concave, supported on the whole $\R^n$,
with globally Lipschitz potential, with Lipschitz constant $L$. 
Let $\gamma$ be the Gaussian measure centered at a point $x_0$ 
and with covariance $\frac n {L^2} Id$. Then 
\[
\log \chi^2 ( \gamma \mid \mu ) \leq n (1+\sigma_0) 
+ \frac n2 \log \left( \frac{L^2 C_P}{n} \right) ,
\]
where $C_P$ is the Poincar\'e constant of $\mu$.   
\end{lem}
In particular when $\sigma_0 = O (1)$ and 
all other parameters of the problem 
are at most polynomial in $n$, we get 
$\log \chi_2^2 ( \gamma \mid \mu ) \leq O ( n \log n )$. 
With this choice of a warm start, 
and observing that in the unconstrained case the 
parameter $A$ is just $O ( L / \sqrt n )$, 
the previous theorem gives 
$\frac 1n W_1^2 ( x_k , \mu )\leq \epsilon$
after 
\[
k = \Theta^* \left( \frac{C_P^3 L^2 n^2}{\epsilon^2} \right) 
\]
steps, with $\eta$ chosen appropriately. 
Also in the constrained case, one can get  
a non trivial complexity estimate from Theorem~\ref{thm:poincare}
by choosing the uniform measure on a ball contained in the support 
as a warm start. We leave this annoying computation to the reader. 
 
Finally let us point out that it is also possible 
to obtain interesting bounds from our result 
when the potential is not globally Lipschitz, 
simply by restricting the measure to a large ball. 
For simplicity let us only spell out the argument when 
the measure $\mu$ is supported on the whole $\R^n$ 
and when we have a linear control on the gradient of the 
potential, but the method could give non trivial bounds in 
more general situations. So let $\mu$ be a log-concave measure supported 
on the whole $\R^n$, let $\phi$ be its potential, 
and consider the Langevin algorithm associated to 
the measure $\mu$ conditioned on the ball of radius $R$:
\begin{equation}\label{eq:PLMCball}
x_{k+1} = \mathcal P \left( x_k + \sqrt \eta \xi_{k+1} 
- \frac \eta 2 \nabla \phi ( x_k ) \right) , 
\end{equation}
where $\mathcal P$ is the orthogonal projection on the ball of radius $R$: 
\[
\mathcal P (x) =
\begin{cases}
x & \text{if } \vert x\vert \leq R \\
\frac{R x}{\vert x\vert} & \text{otherwise} 
\end{cases}
\]
In this special case, Theorem~\ref{thm:logsob} 
yields the following complexity for the Langevin algorithm.  
\begin{thm}\label{thm:largeball}
Assume that $\mu$ is log-concave, supported on the whole $\R^n$ 
and that the gradient of the potential $\phi$ grows at most linearly:
\[
\vert \nabla \phi (x) \vert \leq \beta ( \vert x\vert +1) ,
\] 
for all $x\in\R^n$ and for some $\beta >0$. 
Assume that the Langevin algorithm 
is initiated at $0$, that $\sigma_0 = O (1)$, 
that $\int \vert x\vert^2 \, d\mu = O (n)$,
and let $C_{LS}$ be the log-Sobolev constant of $\mu$, 
with the convention that it equals $+\infty$ if $\mu$ does 
not satisfy log-Sobolev. 
Then choosing $R = \Theta^* ( \sqrt n )$, 
$\eta = \Theta^* \left( \epsilon^2 \max ( \beta ,1 )^{-2} \min ( C_{LS} , n )^{-2} \right)$
and running the algorithm~\eqref{eq:PLMCball} initiated at $0$ for 
\[ 
k = \Theta^* \left( \frac { \min ( C_{LS} , n)^3 \max(\beta,1)^2 } { \epsilon^2 }  \right) 
\]
steps produces a point $x_k$ satisfying $\frac 1n W_2^2 ( x_k , \mu ) \leq \epsilon$. 
\end{thm}
Note in particular that in the case where $C_{LS}$ and $\beta$ 
are of constant order the complexity does not depend on the dimension.  
\medskip

\paragraph{Related works.} 
We end this introduction with a discussion on a short selection 
of related works. 
Let us first mention that as far as we know, the
Langevin algorithm with support constraint 
was only investigated in our previous paper 
with Bubeck and Eldan~\cite{BEL}. In this paper 
the potential was assumed to be gradient Lipschitz.  
In all the works that we could find on the 
Langevin Monte Carlo algorithm the potential is always assumed to be 
smooth, most of the time gradient Lipschitz. 
This hypothesis is somewhat relaxed in the recent article~\cite{CDJB}, 
but the authors analyze the Langevin algorithm for a smoothed out approximation of $\mu$, 
and in any case they still require the gradient of the potential to be H\"older
continuous. The present work appears to be the first were
$\nabla \phi$ is allowed to be discontinuous. 

Let us give the state of the art convergence bounds for 
the Langevin algorithm in the smooth, unconstrained case. 
The first quantitative result appears to be 
Dalalyan's article~\cite{dalalyan}. The 
result is in total variation distance rather than Wasserstein
but as in the present work the strategy consists in writing 
\begin{equation}\label{eq:triangle}
TV ( x_k , \mu ) \leq TV ( x_k , X_{k\eta} ) + TV ( X_{k\eta} , \mu  ) 
\end{equation}
and estimating both terms separately. 
His assumption is that the potential $\phi$ satisfies 
\[
\alpha Id \leq \nabla^2 \phi \leq \beta Id
\] 
pointwise on the whole $\R^n$, where $\alpha$ and $\beta$
are positive constants. Actually a closer look at his argument shows that 
he does not really use log-concavity. Indeed, his
main contribution is a bound
for the relative entropy of the Langevin algorithm at time $k$ with 
respect to the corresponding point in the Langevin diffusion. 
That part of the argument 
is a nice application of Girsanov's formula and does not use log-concavity at all,   
only the fact that $\nabla \phi$ is Lipschitz is needed. 
Dalalyan only uses strict log-concavity to estimate how fast 
the diffusion $(X_t)$ converges to $\mu$. But that only 
requires Poincar\'e for an exponentially fast decay in chi-square divergence
or log-Sobolev for a decay in relative entropy. Dalalyan's theorem can thus 
be rewritten as follows: if $d\mu = \e^{-\phi} \,dx $ is supported on the whole $\R^n$, 
if $\nabla \phi$ is Lipschitz with constant $\beta$ and if $\mu$
satisfies the log-Sobolev inequality with constant $C_{LS}$ then 
after $k$ steps of the Langevin algorithm with times step parameter $\eta$ 
we have 
\[
TV ( x_k , \mu ) 
\leq D ( x_0 \mid \mu )^{1/2} \e^{- k\eta / 2 C_{LS}} 
+\beta n^{1/2} ( 1 + \E [ \sigma_0 ] )^{1/2}  k^{1/2} \eta  , 
\]
where again $\sigma_0 = \frac 1n \left( \phi(x_0) - \min_{\R^n} \{ \phi \} \right)$. 
The result depends on a warm start hypothesis, the algorithm must be initiated from a 
random point $x_0$ whose relative entropy to the target measure is finite. On the other 
hand, it is not hard to see that one can find a Gaussian measure whose relative 
entropy to $\mu$ is $O^* (n)$. As a result, it follows from the previous bound that 
if $\eta$ is chosen appropriately then one has $TV ( x_k  ,\mu ) \leq \epsilon$ after 
\[
k = \Theta^* \left( \frac { C_{LS}^2 \beta^2 n } {\epsilon^2 } \right) 
\]
steps of the algorithm. 

Durmus and Moulines~\cite{DM} have the same set of hypothesis 
as Dalalyan but they prove a result in Wasserstein distance rather than total variation. 
As opposed to Dalayan they really 
use the hypothesis $\nabla ^2 \phi \geq \alpha Id$ for some positive $\alpha$.  
Also their approach is a bit different from that of Dalalyan:  
instead of bounding $W_2 ( x_k, X_{k\eta} )$ and $W_2 ( X_{k\eta} , \mu)$ separately 
they directly obtain a recursive inequality for $W_2 ( x_k , \mu )$. 
Their approach essentially yields the following result: 
Suppose that $\alpha Id \leq \nabla^2 \phi \leq \beta Id$ pointwise on the whole $\R^n$
for some positive constants $\alpha,\beta$. Assume also that the time step parameter 
$\eta$ satisfies $\eta \leq \frac 1 {2\beta}$. Then 
\begin{equation}\label{eq:DM}
W_2 ( x_k , \mu ) 
\leq \left(1 - \frac {\alpha\eta} 2\right)^k W_2 ( x_0 , \mu ) 
+ \frac {2\beta} \alpha n^{1/2} \eta^{1/2} . 
\end{equation}
Actually, the result of Durmus and Moulines is a bit more involved, for a (short) proof 
of that very statement see~\cite[Theorem~1.1]{dalalyan2}. 

This result implies that $\frac 1n W_2^2 ( x_k , \mu ) \leq \epsilon$
after a number of steps $k = O^* \left( \frac{\beta^2}{\alpha^3\epsilon} \right)$,
with time step parameter of order $\epsilon \alpha^2 / \beta^2$. 
This should be compared to the complexity given by Theorem~\ref{thm:largeball}
in this case. Indeed, observe that the hypothesis $\alpha Id \leq \nabla^2 \phi \leq \beta Id$ 
implies that the log-Sobolev constant is $1/\alpha$ at most and that $\nabla \phi$ 
grows linearly. Therefore Theorem~\ref{thm:largeball} applies, and it gives 
the following complexity: $k=\Theta^* \left( \frac{\beta^2}{\alpha^3 \epsilon^2} \right)$.
The dependence in $\epsilon$ is thus worse ($\epsilon^2$ rather $\epsilon$) 
but the dependence in the other parameters is the same, which is quite remarkable 
given the fact that Theorem~\ref{thm:largeball} holds 
under considerably weaker assumptions. 

Lastly, let us also mention Vempala and Wibisono's work~\cite{VW} whose approach is 
similar in spirit to that of Durmus and Moulines but gives 
a result closer to Dalalyan's. They prove that if 
$\nabla \phi$ is Lipschitz with constant $\beta$, 
if $\mu$ satisfies log-Sobolev with constant $C_{LS}$,
and if the time step parameter satisfies $\eta \leq 1 / (4C_{LS}\beta^2)$ 
then after $k$ steps of the algorithm one has  
\begin{equation}\label{eq:VW}
D ( x_k \mid \mu ) \leq \e^{- k\eta / C_{LS} } D( x_0 \mid \mu ) 
+ 8 n \beta^2 C_{LS} \eta .
\end{equation}
As in the result of Dalalyan the measure $\mu$ is not assumed to be log-concave, 
only log-Sobolev is required. Let us note that this result 
recovers Dalalyan's by Pinsker's inequality.
Let us also remark that combining it with the transport inequality 
$W_2^2 ( x_k , \mu ) \leq 2C_{LS} \,D( x_k \mid \mu )$,
which is a consequence of log-Sobolev, one gets 
\[
W_2^2 ( x_k , \mu ) \leq 2 C_{LS} D(x_0\mid \mu) \e^{- k\eta / C_{LS} } 
+ 16 \, n \, \beta^2 C_{LS}^2 \eta . 
\]
This pretty much recovers~\eqref{eq:DM} under a weaker hypothesis: Log-Sobolev rather than uniform 
convexity of the potential, with two caveats: The hypothesis on the time step 
parameter is a bit more restrictive, and this is from a warm start 
in the relative entropy sense. 
\medskip

\paragraph{Acknowledgments.} The author is grateful to S\'ebastien Bubeck 
and Ronen Eldan for a number of useful discussions related to this work.
We are also grateful to Andre Wibisono who brought to our attention
the $W_2$/chi-square inequality used in the proof of Theorem~\ref{thm:poincare}. 
In the first version of the paper the formulation of that theorem was slightly weaker, 
with $W_1$ in place $W_2$. 

\section{The Langevin diffusion with reflected boundary condition}

In section we define properly the Langevin diffusion with reflection
at the boundary of $K$:
\[
d X_t = d B_t - \frac 12 \nabla \phi(X_t) - d \Phi_t .  
\]
Recall that $d\mu = \e^{-\phi} dx$ is assumed to be log-concave,  
which means that $\phi\colon \R^n \to \R\cup\{+\infty\}$ is convex. 
Actually, it will be slightly more convenient 
for our purposes to assume that the domain of $\phi$ is the whole $\R^n$ and that the measure $\mu$ 
is given by $\mu(dx) = \mathbf 1_K (x) \e^{-\phi(x)} \, dx$ where $K$ is a convex subset of 
$\R^n$ with non empty interior. Since we do not assume the potential $\phi$ to be 
everywhere differentiable, the expression $\nabla \phi ( X_t)$ needs to be clarified. 
Let us agree on the convention 
that in the sequel $\nabla \phi (x)$ stands for the element of the subdifferential 
of $\phi$ at point $x$ whose Euclidean norm is minimal. Since $\phi$ is assumed to 
be a convex function whose domain is the whole $\R^n$, for every $x\in\R^n$ the subdifferential 
of $\phi$ at $x$ is a non empty closed convex set, so that the Euclidean norm does uniquely attain 
its minimum on this set.  


According to Tanaka~\cite[Theorem~3.1]{tanaka}, given a continuous semi-martingale $(W_t)$
taking values in $\R^n$ and satisfying $W_0\in K$, 
there exists a unique couple $(X_t,\Phi_t)$ of continuous semi-martingales such that, almost surely 
\begin{enumerate}
\item $X_t \in K$ for all $t\in\R_+$,
\item $X_t = W_t - \Phi_t$ for all $t\in\R_+$,
\item $(\Phi_t)$ is of the form  $\Phi_t = \int_0^t \nu_s \, d\ell_s$ where $\ell$ is a measure on 
$\R_+$ which is finite on bounded intervals and supported on the set 
$\{t\in\R_+\colon X_t\in\partial K\}$, and for any such $t$ the vector 
$\nu_t$ is an outer unit normal to the boundary of $K$ at $X_t$.  
\end{enumerate} 
In words the process $(\Phi_t)$ is a finite variation and continuous process 
which repels $(X_t)$ inwards when it reaches the boundary of $K$. In the sequel we shall say that the 
process $(X_t)$ is the \emph{reflection} of $(W_t)$ at the boundary of $K$ and that 
the process $(\Phi_t)$ is \emph{associated} to $(X_t)$. 
The process $(\ell_t)$ is called the \emph{local time} of $(X_t)$ at the boundary of $K$. 

Now given a standard Brownian motion $(B_t)$ and a starting point $x\in K$, 
we want to argue that there exists a unique process $(X_t)$ such that 
$(X_t)$ is the reflection at the boundary of $K$ of the semi-martingale 
\[
x + B_t - \frac 12 \int_0^t \nabla \phi (X_s)\, ds . 
\] 
In other words we want 
\[
d X_t = d B_t -  \frac 12 \nabla \phi (X_t)\, dt - d \Phi_t ,  
\]
where $(\Phi_t)$ is associated to $(X_t)$. This is a stochastic differential 
equation with reflected boundary condition. 
If $\nabla \phi$ is Lipschitz continuous then again Tanaka~\cite[Theorem~4.1]{tanaka} shows that 
this equation admits a unique strong solution. Let us now explain why Tanaka's result remains valid
in the present context, even though $\nabla \phi$ is not assumed to be continuous. 

First of all, notice that since $\nabla \phi$ is the gradient of a convex function, 
it is a monotone map, in the sense that 
\[
\langle x-y,\nabla \phi (x) - \nabla \phi (y)\rangle \geq 0 , \quad \forall x,y\in\R^n.
\]
This property immediately implies pathwise uniqueness for 
the equation. Indeed suppose that $X$ and $\tilde X$ are two solutions
of the equation and that $\Phi$ ant $\tilde\Phi$ are the associated processes. Then 
\[
d \vert X_t - \tilde X_t\vert^2 = - \frac 12 \langle X_t - \tilde X_t , 
 \nabla \phi ( X_t ) - \nabla \phi ( \tilde X_t ) \rangle \, dt 
- \langle X_t - \tilde X_t , d\Phi_t \rangle  
-  \langle \tilde X_t - X_t , d \tilde \Phi_t\rangle . 
\]
The first term of the right-hand side is non positive by 
monotony of $\nabla \phi$. The second term is also non positive. 
Indeed the fact that $\Phi$ is associated to $X$ and $\tilde X$ takes values in $K$ 
imply that $\langle X_t - \tilde X_t , d \Phi_t\rangle \geq 0$ for all $t$. Similarly 
$\langle \tilde X_t - X_t , d \tilde \Phi_t\rangle \geq 0$. Therefore the quantity 
$\vert X_t - \tilde X_t\vert$ is almost surely non increasing, which 
obviously implies that the equation has the pathwise uniqueness property. 
To get existence, one option is to approximate $\phi$ by a smooth 
convex function and pass to the limit, as in~\cite{cepa}. That paper 
is a little involved and is written French so let us give an alternative 
argument for completeness. This argument only works when $\nabla \phi$ 
is bounded, but this is the only case we shall consider here.
When $\nabla \phi$ is bounded, it is well known that an 
application of Girsanov yields the existence of a solution to the 
equation. Indeed, let $X$ be the reflection at the boundary of $K$ of the 
process $x+B$. We have 
\[
d X_t = d B_t - d \Phi_t , 
\]
where $\Phi$ is associated to $X$. Since $\nabla \phi$ is bounded the process 
$(D_t)$ given by 
\[
D_t = \exp \left( - \frac 12 \int_0^t \langle \nabla \phi (X_s) , d B_s \rangle 
- \frac 18 \int_0^t \vert \nabla \phi ( X_s ) \vert^2 \, ds \right) 
\]
is a positive martingale with expectation $1$. If we fix a time horizon $T>0$ and define a 
new probability measure by $d \mathbb Q = D_T \, d\mathbb P$, then by 
Girsanov, the process $(\tilde B_t)_{t\in [0,T]}$ given by 
\[
\tilde B_t = B_t + \frac 12 \int_0^t \nabla \phi ( X_s ) \, ds ,  
\]  
is a standard Brownian motion under the new measure $\mathbb Q$. 
Since 
\[
d X_t = d\tilde B_t - \frac 12 \nabla \phi ( X_t) \, dt + d \Phi_t , 
\]
where $\Phi$ is associated to $X$ this shows that under $\mathbb Q$ the 
process $X$ solves the equation driven by $\tilde B$.  
This proves \emph{weak} existence of a solution, in the sense 
that we had to change the probability space and the Brownian motion. 
However it is well known that weak existence and pathwise uniqueness altogether 
imply strong existence, see for instance~\cite[Chapter IV, Theorem~1.1]{IW}. 
Strictly speaking this only shows strong existence on a finite time interval 
$[0,T]$. But we can eventually let $T$ tend $+\infty$ and use pathwise uniqueness again
to get strong existence of a solution defined for all time. Details are left to the reader.  

The solution $(X_t)$ of the equation is a Markov process, 
whose semigroup is denoted $(P_t)$ in the sequel: 
For every test function $f\colon \R^n \to \R$ and every $x\in K$ 
\[
P_t f(x) = \E_x [f(X_t) ] 
\]
where the subscript $x$ next to the expectation denotes the starting point of $X_t$. 
By It\^o's formula, if $f$ is $\mathcal C^2$-smooth in a neighborhood of $K$ then
\[
d f(X_t) = \langle \nabla f(X_t) , dB_t\rangle
 - \frac 12 \langle \nabla f(X_t) , \nabla \phi ( X_t ) \rangle \, dt 
 + \frac 12 \Delta f(X_t) \, dt - \langle \nabla f(X_t) , d \Phi_t \rangle. 
\]
Here we are using It\^o's formula for a continuous semi-martingale having a 
finite variation part, see for instance~\cite[Chapter IV, Corollary 32.10]{RW}. 
If $f$ satisfies the Neumann boundary condition: 
$\langle \nabla f(x) , \nu \rangle$ for every $x\in\partial K$ 
and every $\nu$ that is normal to the boundary of $K$ at $x$ then the last term 
of the right-hand side vanishes. 
Taking expectation we then see that the generator of the semigroup $(P_t)$ is 
\[
\frac 12 Lf := \frac 12 \left( \Delta f - \langle \nabla f , \nabla \phi \rangle \right) 
\]
with Neumann boundary condition. Also, an integration by part then shows that 
\[
\int_K (L f) g \, d\mu = - \int_K \langle \nabla f , \nabla g\rangle \, d\mu ,  
\]  
for every $f,g$ in the domain of $L$. In particular the operator $L$ is symmetric 
in $L^2 ( \mu)$, which implies that $\mu$ is a reversible measure for the semigroup $(P_t)$. 
\section{Discretization of the Langevin diffusion}
In this section we prove Theorem~\ref{thm:main}. This is the main contribution of the 
article. We begin with a bound on the local time $(\ell_t)$ of the diffusion $(X_t)$ 
at the boundary of $K$. We need to show that $\ell_t = O (t)$. That lemma is essentially taken from
our previous work with Bubeck and Eldan~\cite{BEL}, except that we have simplified the proof 
and improved the result a bit. 
\begin{lem}\label{lem:Lt}
Assume that the Langevin diffusion $(X_t)$ is initiated 
at point $x_0$ in the interior of $K$ and recall the definition~\eqref{eq:sigmar} 
of $r_0$ and $\sigma_0$. Then for every $t>0$ we have
\[
\E [ \ell_t^2 ]^{1/2} \leq \frac {n (1+\sigma_0) t }{r_0} . 
\]
\end{lem}
\begin{proof}
By It\^o's formula we have
\[
d \vert X_t - x_0 \vert^2 = 2 \langle X_t -x_0, d B_t \rangle 
- \langle X_t -x_0, \nabla \phi (X_t) \rangle d t 
- 2 \langle X_t - x_0 , d \Phi_t \rangle + n \, dt  .
\]
Recall that $d \Phi_t = \nu_t \, d \ell_t$ where $\nu_t$ is an outer unit normal at $X_t$. 
By definition of $(\Phi_t)$, $r_0$ and $\ell_t$ we have 
\[
\langle X_t - x_0 , d\Phi_t \rangle 
\geq \sup_{x \in K } \langle x -x_0 , d\Phi_t \rangle 
\geq r_0 d \ell_t .   
\]
Also by convexity of $\phi$
\[ 
- \langle X_t - x_0 , \nabla \phi ( X_t ) \rangle \leq \phi (x_0) - \phi ( X_t ) \leq n \sigma_0. 
\]
We thus obtain
\begin{equation}\label{eq:stepell}
\vert X_t -x_0 \vert^2 
+ 2 r_0 \ell_t \leq n(1+\sigma_0)t + \int_0^t \langle X_s -x_0 , dB_s \rangle . 
\end{equation}
Taking expectation already gives a bound on the first moment
of $\ell_t$. To get a bound on the second moment observe 
that~\eqref{eq:stepell} implies that 
\[
4 r_0^2 \E [  \ell_t^2 ]  \leq  n^2 (1 + \sigma_0)^2 t^2 
+ \E \left[ \left( \int_0^t \langle X_s-x_0,dB_s\rangle \right)^2 \right] . 
\]
By It\^o's isometry and using~\eqref{eq:stepell} again, 
this time to bound $\E [ \vert X_t-x_0 \vert^2 ]$, 
we get
\[
\E \left[ \left( \int_0^t \langle X_s-x_0,dB_s\rangle \right)^2 \right]
= \E \left[ \int_0^t \vert X_s -x_0 \vert^2  \, ds \right] \leq 
 n (1+\sigma_0) \frac {t^2} 2 .  
\]
Plugging this into the previous display yields the 
desired inequality.  
\end{proof}
We also need the following elementary bound on the maximum 
of Gaussian vectors. We provide a proof for completeness.
\begin{lem}\label{lem:max}
Let $G_1,\dotsc,G_k$ be standard Gaussian vectors on $\R^n$. Then 
\[
\E \left[ \max_{i\leq k} \left\{ \vert G_i \vert^2 \right\} \right]
\leq \e ( n + 2 \log k ).   
\]
\end{lem}
\begin{proof}
Set $\chi_i = \vert G_i \vert^2$ for every $i$. The 
$p$-th moment of $\chi_i$ satisfies 
\[
\E [\chi_i^p ] 
= \frac{ 2^p \Gamma \left( \frac n2 + p \right) }{ \Gamma \left( \frac n2 \right) } 
\leq (n + 2 (p-1) )^p  , 
\] 
at least when $p$ is an integer. Therefore
\[
\E \left[\max_{i\leq k} \{ \chi_i \} \right]
\leq \left[\sum_{i \leq k} \E [\chi_i^p ] \right]^{1/p} 
\leq k^{1/p} (n + 2 (p-1) ) 
\]
Choosing $p$ to be the smallest integer larger than $\log k$ 
yields the result. 
\end{proof}
We are now in a position to prove 
the main result. 
\begin{proof}[Proof of Theorem~\ref{thm:main}]
We first couple the diffusion $(X_t)$ and its discretization $(x_k)$ 
in the most natural way one could think of, by choosing 
the sequence $(\xi_k)$ as follows: 
\begin{equation}\label{eq:coupling}
\xi_{k} = B_{k \eta } - B_{(k-1)\eta }  , \quad k \geq 1. 
\end{equation}
Observe that for any $x\in K$ and $y\in\R^n$ 
we have $\vert x - \mathcal P_K (y) \vert \leq \vert x- y\vert$. 
Therefore
\[
\begin{split}
\vert X_{(i+1)\eta} -  x_{i+1} \vert^2
& = \left\vert  X_{(i+1)\eta} - \mathcal P \left( x_i + \xi_{i+1} 
- \frac \eta 2 \nabla \phi (x_i) \right) \right\vert^2  \\
& \leq \left\vert  X_{(i+1)\eta} - x_i - \xi_{i+1} 
+ \frac \eta 2 \nabla \phi (x_i) \right\vert^2 . 
\end{split}
\]
Let $(\widetilde X_t)$ be the process defined by
\[
\widetilde X_t = x_{i} + B_t - B_{i\eta} -\frac{t-i\eta}2 \nabla \phi (x_i )   
\]
for all $t$ between $i\eta$ and $(i+1)\eta$. Then $\widetilde X_{i\eta} = x_i$ and 
the previous display can be rewritten as 
\[
\vert X_{(i+1)\eta} -  x_{i+1} \vert^2 \leq 
\vert X_{(i+1)\eta} - \widetilde X_{(i+1)\eta} \vert^2 . 
\]
The process $(X_t-\widetilde X_t)$ 
is continuous with finite variation on $[i\eta , (i+1)\eta]$ 
(the Brownian part cancels out). Therefore, on that interval 
we have
\[
\begin{split}
d \vert X_t - \widetilde X_t \vert^2 
& = 2 \langle X_t - \widetilde X_t , d X_t - d \widetilde X_t \rangle \\
& = - \langle X_t - \widetilde X_t , \nabla \phi (X_t) 
- \nabla \phi ( x_i ) \rangle \, dt - 2 \langle X_t - \widetilde X_t ,  d \Phi_t \rangle . 
\end{split}
\]
Again by monotony of $\nabla \phi$ we have 
$\langle X_t - x_i , \nabla \phi (X_t) - \nabla \phi ( x_i ) \rangle \geq 0$.
Also since $x_i \in K$ and $(\Phi_t)$ is associated to $(X_t)$ we have 
$\langle X_t - x_i , d \Phi_t \rangle \geq 0$. Plugging this back in the previous 
display yields
\[
d \vert X_t - \widetilde X_t \vert^2 
\leq \langle \widetilde X_t -x_i , (\nabla \phi ( X_t ) -\nabla \phi ( x_i ) ) \, dt 
+ 2 d\Phi_t \rangle . 
\]
Now we replace $\widetilde X_t$ by its value, and we integrate 
between $i\eta$ and $(i+1)\eta$. We get 
\[
\begin{split}
\vert X_{(i+1)\eta} - x_{i+1} \vert^2
& \leq \vert X_{i\eta} - x_i \vert^2 \\
& + \int_{i\eta}^{(i+1)\eta} 
\langle B_t - B_{i\eta} - \frac{t-i\eta}2 \nabla \phi (x_i ) , 
( \nabla \phi ( X_t ) - \nabla \phi ( x_i ) ) \, dt 
+ 2 d\Phi_t \rangle .
\end{split} 
\]
We now take expectation. Note that the martingale property 
of the Brownian motion implies that 
$\E [ \langle B_t - B_{i\eta} , \nabla \phi ( x_i ) \rangle ] = 0$ 
and that 
\[
\E [ \langle B_t - B_{i\eta} , d \Phi_t \rangle ] = 
 \E [\langle B_{(i+1)\eta} - B_{i\eta} , d \Phi_t \rangle ] . 
\]
We also use the hypothesis $\vert\nabla \phi \vert \leq L$ and 
the inequality $\E [ \vert B_t - B_{i\eta} \vert ]\leq n^{1/2} (t-i\eta)^{1/2}$. 
We obtain 
\begin{equation}\label{eq:something}
\begin{split}
\E \left[\vert X_{(i+1)\eta} - x_{i+1} \vert^2 \right]
& \leq \E \left[\vert X_{i\eta} - x_i \vert^2 \right]
+ \frac23 L \eta^{3/2} n^{1/2} \\ 
& + 2 \,\E \left[\langle \xi_{i+1} , \Phi_{(i+1)\eta} - \Phi_{i\eta} \rangle \right]
+ \frac 12 L^2 \eta^2 
+ L\eta\, \E [\ell_{(i+1)\eta} - \ell_{i\eta} ] ,
\end{split}
\end{equation}
where $\xi_{i+1} = B_{(i+1)\eta} - B_{i\eta}$. 
By Cauchy-Schwarz and Lemma~\ref{lem:max} 
\[
\begin{split} 
\E \left[ \sum_{i=0}^{k-1} \langle \xi_{i+1} ,  \Phi_{(i+1)\eta} - \Phi_{i\eta} \rangle  \right]
& \leq \E \left[ \max_{i\leq k} \left\{ \vert \xi_{i} \vert \right\} 
\, \ell_{k\eta} \right] \\
& \leq \e^{1/2} (n+2 \log k)^{1/2} \eta^{1/2} \E [ \ell_{k\eta}^2 ]^{1/2} . 
\end{split}
\]
Summing~\eqref{eq:something} over $i$ thus yields 
\[
\begin{split}
\E \left[ \vert X_{k\eta} - x_k \vert^2 \right]
& \leq  2 \e^{1/2} (n+2 \log k)^{1/2} \eta^{1/2} \E [ \ell_{k\eta}^2 ]^{1/2}  \\
& + L \eta \, \E [ \ell_{k\eta} ]  +  \frac 23 L n^{1/2} k \eta^{3/2} + \frac 12 L^2 k \eta^2 . 
\end{split}
\]
Now recall from Lemma~\ref{lem:Lt} that 
\[
\E [ \ell_{k\eta}^2 ]^{1/2} \leq \frac{ n (1+\sigma_0) k\eta } {r_0}.   
\]
Lastly, we use the assumption $\eta < n / L^2$ to simplify the inequality a 
bit. We finally obtain 
\[
\E \left[ \vert X_{k\eta} - x_k \vert^2 \right]
\leq (2\e^{1/2} +1)(n+2\log k)^{1/2} \frac{n(1+\sigma_0)}{r_0} k \eta^{3/2} 
+ \frac 76 L n^{1/2} k \eta^{3/2} , 
\]
which is the result. 
\end{proof}
\section{Convergence of the algorithm under log-Sobolev} 
In this section we prove Theorem~\ref{thm:logsob}.
Observe first that the Wasserstein distance is indeed a distance, 
so it satisfies the triangle inequality and we have
\[
\frac 1n W_2^2 ( x_k , \mu ) 
 \leq \frac 2n W_2^2 ( x_k , X_{k\eta} ) + \frac 2n W_2^2 ( X_{k\eta} , \mu ) . 
\]
The first term of the right-side is handled by Theorem~\ref{thm:main}, 
so we only need to bound the second term. This is the purpose 
of the following lemma. In this lemma $(P_t)$ stands for the semi-group
of the Langevin diffusion~\eqref{eq:tanaka-diff}. In other 
words $\nu P_t$ denotes the law of $X_t$ when $X_0$ has law $\nu$. 
\begin{lem}
Assume that $\mu$ is log-concave with globally Lipschitz 
potential on its support, with Lipschitz constant $L$. Assume
also that $\mu$ satisfies log-Sobolev with constant $C_{LS}$. 
Then for every $x_0$ in the interior of the support of $\mu$
and every $t>0$ we have 
\[
\frac 1n W_2^2 ( \delta_{x_0} P_t , \mu ) 
\leq 4 C_{LS} \left( 1 + \log \left( \frac{ \max ( C_{LS} ,1 )\, n }{ \min ( r_0 , 1 ) } \right)   
+ \sigma_0 + \frac L n \right) \e^{- t / 2 C_{LS}} ,  
\]
where again the parameters $\sigma_0$ and $r_0$ are defined by~\eqref{eq:sigmar}. 
\end{lem}
\begin{proof}
If $\mu$ satisfies the logarithmic Sobolev inequality then 
it satisfies the transport inequality:
\[
W_2^2 ( \nu ; \mu ) \leq 2 C_{CLS} D( \nu \mid \mu)  
\]
for every probability measure $\nu$. This is due to Otto and 
Villani~\cite{OV}, see also~\cite{BoGL}. The log-Sobolev
inequality also implies that the relative entropy
decays exponentially fast along the semigroup $(P_t)$: 
\[
D (\nu P_t \mid \mu ) \leq \e^{ - t / C_{LS} }  D ( \nu \mid \mu ). 
\]
This is really folklore, one just need to observe that the derivative 
of the entropy is the Fisher information, and combine log-Sobolev with a
Gronwall type argument. Combining the two inequalities yields
\[
W_2^2 ( \nu P_t , \mu ) \leq 2 C_{LS}  \e^{ - t / C_{LS} } D ( \nu \mid \mu ) . 
\]
This cannot be applied directly to a Dirac point mass. However, 
observe that the convexity of $\phi$ implies that the Wasserstein 
distance is non increasing along the semigroup: For any two probability
measures $\nu_0,\nu_1$ and every time $t$ we have 
\[
W_2^2 ( \nu_0 P_t , \nu_1 P_t ) \leq  W_2^2 ( \nu_0 , \nu_1 ).  
\]
This is a well-known fact, which is easily seen using parallel coupling.
See the proof of the pathwise uniqueness property in section 2. 
Combining this with the triangle inequality for $W_2$ we thus get 
\begin{equation}\label{eq:stepLS}
\begin{split}
W_2^2 ( \delta_{x_0} P_t , \mu ) 
& \leq 2 W_2^2 ( \delta_{x_0} P_t , \nu P_t ) + 2 W_2^2 ( \nu P_t , \mu ) \\
& \leq 2 W_2^2 ( \delta_{x_0},\nu ) + 4 C_{LS} \e^{ - t / C_{LS} } D ( \nu \mid \mu ). 
\end{split}
\end{equation}
This is valid for every $\nu$ and it is natural to take
$\nu$ to be the measure $\mu$ conditioned to the ball $B(x_0,\delta)$ 
for some small $\delta >0$. Then $W^2_2 ( \delta_{x_0} , \nu ) \leq \delta^2$ 
and 
\[
D( \nu  \mid \mu ) 
= \log \left( \frac 1 { \mu ( B(x_0, \delta) ) } \right) . 
\]
If $\delta \leq r_0$ then $B(x_0,\delta)$ is included in the support of $\mu$ and 
we have  
\[
\log \left( \frac 1 { \mu ( B(x_0 , \delta) ) } \right) \leq 
\max_{B(x_0,\delta)} \{ \phi \} + n \log \left( \frac 1 \delta \right) 
+ \log \left( \frac 1 {v_n} \right) , 
\]
where $v_n$ is the Lebesgue measure of the unit ball in dimension $n$. 
Recall that $\log \left( \frac 1 {v_n} \right) \leq n \log n $. 
Also
\[
\max_{B(x_0,\delta)} \{ \phi \} \leq \phi(x_0) + L\delta 
\leq \min_{K} \{ \phi \} + n \sigma_0 + L \delta. 
\]
Moreover
\[
\min_K \{ \phi \} 
\leq \int_{\R^n} \phi \, d\mu = S(\mu) , 
\] 
where $S$ denotes the Shannon entropy. It is well-known 
that among measures of fixed covariance the Gaussian measure 
maximizes the Shannon entropy (this is just Jensen actually). 
Therefore 
\[
\begin{split}
S ( \mu ) 
& = \frac n2 \log (2\pi e) + \frac 12 \log \mathrm{det} ( \mathrm{cov} ( \mu ) )  \\
& \leq \frac n 2  \log (2\pi e C_{LS})  . 
\end{split}
\]
The last inequality is just a consequence 
of the fact that the log-Sobolev inequality implies Poincar\'e, 
which in turn implies a bound on the covariance matrix. 
Plugging everything back in~\eqref{eq:stepLS} we get 
\[
\frac 1n W_2^2 ( \delta_{x_0} P_t , \mu )
\leq \frac {2 \delta^2}n 
+  4 C_{LS} \left( \frac 32 + \log \left( \frac n \delta \right) 
+  \sigma_0 + \frac Ln \right) \e^{- t / C_{LS}}
\]
for every $\delta \leq \min (r_0,1)$. 
Choosing $\delta = \min \left( (2 n C_{LS})^{1/2} \e^{-t/2C_{LS}} , r_0 , 1 \right)$ 
and using the inequality $x\e^{-x} \leq \e^{-x/2}$ yields the result. 
\end{proof}
\section{A convergence result using Poincar\'e only}\label{sec:poincare}
In this section we prove Theorem~\ref{thm:poincare}. 
Again the idea is to write 
\[
\frac 1n W_2^2 ( x_k , \mu ) 
\leq \frac 2n W_2^2 (x_k , X_{k\eta} ) + \frac 2n W_2^2 ( X_{k\eta} , \mu ) , 
\]
and to bound the first term using Theorem~\ref{thm:main}. 
Actually, note that here we allow $x_0=X_0$ to be random, so we rather condition 
on $x_0$, apply Theorem~\ref{thm:main} and then take expectation again. 

Therefore it is enough to bound the second term. This is 
where the Poincar\'e inequality enters the picture. Note 
that this part of the argument does not rely on the log-concavity 
of $\mu$. 
We shall use the following transport/chi-square divergence inequality: 
If $\mu$ satisfies Poincar\'e with constant $C_P$ then for every 
probability measure $\nu$ on $\R^n$ we have 
\[
W_2^2 ( \nu , \mu ) \leq 2 C_P \chi^2 ( \nu \mid \mu ) . 
\]
It seems that this was first proved by Ding~\cite{ding}, with 
a worst constant. The result with constant $2$ is due to 
Liu~\cite{liu}. His argument is combines together the Langevin diffusion and 
the Hamilton-Jacobi semigroup. 

On the other hand it is well-known that 
under Poincar\'e the chi-square divergence decays 
exponentially fast along the Langevin diffusion. 
Letting $(P_t)$ be the semigroup of the Langevin 
diffusion associated to $\mu$ we have 
\[
\chi^2 ( \nu P_t \mid \mu ) \leq \e^{-t / C_P} \chi^2 ( \nu \mid \mu ) . 
\] 
We thus get the following: 
\begin{lem}
Suppose that $\mu$ satisfies Poincar\'e 
with constant $C_P$. Then for every probability 
measure $\nu$ on $\R^n$ and every $t>0$ we have 
\[
W_2^2 ( \nu P_t  , \mu ) \leq 2 C_P \, \chi^2 ( \nu \mid \mu ) \, \e^{-t / C_P} . 
\]
\end{lem}
This finishes the proof of Theorem~\ref{thm:poincare}. 

We end this section with a simple 
estimate of the chi-square divergence 
of an appropriate Gaussian measure 
to $\mu$ in the unconstrained case. 
\begin{proof}[Proof of Lemma~\ref{lem:chi2}]
Recall that $\mu$ is assumed to be supported on the 
whole $\R^n$ with convex and globally Lipschitz potential $\phi$. 
Let $\gamma$ be the Gaussian measure centered at a point $x_0$ and 
with covariance $\alpha Id$ for some $\alpha >0$. Then 
\[
\begin{split}
\chi^2 ( \gamma \mid \mu ) 
& \leq (2\pi\alpha)^{-n} \int_{\R^n} 
\e^{- \frac 1 \alpha \vert x-x_0\vert^2 + \phi (x) } \, dx \\ 
& \leq (2\pi\alpha)^{-n} \int_{\R^n} \e^{-\frac 1 \alpha \vert x-x_0\vert^2 + \phi (x_0) 
+ L \vert x-x_0\vert}  \, dx \\
& \leq (2\pi\alpha)^{-n} \int_{\R^n} \e^{-\frac 1 {2\alpha} \vert x-x_0\vert^2  
+ \phi (x_0) + \frac 12 L^2 \alpha} \, dx 
= (2\pi \alpha)^{-n/2}  \e^{ \phi(x_0) + \frac 12  L^2\alpha} .
\end{split} 
\]
Also, reasoning along the same lines as in the previous section 
we get 
\[
\phi ( x_0 ) \leq \min_{\R^n} \{ \phi \} + n \sigma_0 
\leq \frac n2 \log ( 2\pi \e ) + \frac n2 \log C_P + n \sigma_0  .
\]
Putting everything together and choosing $\alpha = n / L^2$ yields 
\[
\begin{split}
\log \chi^2 ( \gamma \mid \mu ) 
& \leq \frac n2 \left(-\log \alpha + 1 + \log C_P + 2 \sigma_0 \right)
+ \frac 12  L^2 \alpha \\
& = n \left( 1 + \sigma_0 + \frac 12 \log \left( \frac { L^2 C_P } n \right) \right) , 
\end{split} 
\]
which is the result. 
\end{proof}
\section{An extension to the non-globally Lipschitz case}
We begin this section with a simple lemma 
about the Wassertein distance of $\mu$ to $\mu$ restricted to a large ball. 
\begin{lem}
Let $\mu$ be a log-concave measure on $\R^n$, 
and let $\mu_R$ be the measure $\mu$ conditioned on the 
ball centered at $0$ of radius $R$.  
There exists a universal constant $C$ such that
\[
W_2^2 ( \mu , \mu_R ) \leq C M \exp \left( - \frac{ R }{ C \sqrt M} \right) , \quad \forall R \geq C \sqrt{M}
\]
where $M = \int_{\R^n} \vert x\vert^2 \, d \mu$.  
\end{lem}
\begin{proof}
Let $X$ have law $\mu$. Note that by Borell's lemma~\cite[Lemma 3.1]{borell}
we have $\prob ( \vert X\vert \geq t ) 
\leq \e^{ -t / C_0\sqrt M }$ for all $t \geq C_0 \sqrt M$ 
for some universal constant $C_0$. This also implies that 
$\E [ \vert X \vert^4 ] \leq C_1 M^2$ for some $C_1$. Now 
assume that $R \geq C_0 \sqrt M$, let $\tilde X$ have law $\mu_R$ and be independent 
of $X$ and let 
\[
Y = \begin{cases} X & \text{if } \vert X\vert \leq R \\
\tilde X & \text{otherwise} . 
\end{cases}
\]
Then $Y$ also has law $\mu_R$, so that 
\[
\begin{split}
W_2^2 ( \mu,\mu_R ) 
& \leq \E [ \vert X-Y\vert^2 ] = \E [ \vert X-\tilde X\vert^2 ; \vert X\vert > R ] \\
& \leq 4 \E [ \vert X\vert^4 ]^{1/2} \, \prob ( \vert X\vert > R )^{1/2} 
\leq C_1 M \e^{ -  R / (C_0 M^{1/2})} ,  
\end{split}
\]
which is the result. 
\end{proof}
\begin{proof}[Proof of Theorem~\ref{thm:largeball}]
Assuming that $\int_{\R^n} \vert x\vert^2 \, d\mu = O (n)$, 
the previous lemma shows that $\frac 1n W_2^2 ( \mu , \mu_R)$ will be negligible 
as soon as $R$ is a sufficiently large multiple of $\sqrt n \log n$. Now we apply 
Theorem~\ref{thm:logsob} to $\mu_R$ and initiating the Langevin algorithm at $0$. 
Then the parameter $r_0$ is of order $\sqrt n \log n$. Moreover the hypothesis 
\[
\vert \nabla \phi (x) \vert \leq \beta ( \vert x\vert + 1 ) 
\]
show that the potential of $\mu_R$ is Lipschitz 
with constant $O^*(\beta\sqrt n)$. Therefore the constant  
$A$ defined by~\eqref{eq:defA} satisfies $A = O^* ( \max (\beta,1) )$ in this 
case. On the other hand since $\mu_R$ is log-concave and supported on a ball of radius 
$O^*(\sqrt n)$ its log-Sobolev constant is $O^*(n)$ at most. Also, if $\mu$ 
satisfies log-Sobolev, then the log-Sobolev constant of $\mu_R$ cannot be larger 
than a constant factor times that of $\mu$. This follows easily from 
the fact that within log-concave measures log-Sobolev and 
Gaussian concentration are equivalent, see \cite[Theorem 1.2]{milman-duke}. 
To sum up, the log-Sobolev constant of $\mu_R$ is $O^* ( \min ( n ,C_{LS} ) )$
where $C_{LS}$ is the log-Sobolev constant of $\mu$ (which is possibly infinite). 
Applying Theorem~\ref{thm:logsob} we see that after
\[
k = \Theta^* \left( \frac{ \min ( C_{LS} , n )^3 \max ( \beta , 1 )^2 } {\epsilon^2} \right) 
\] 
of the Langevin algorithm for $\mu_R$ initiated at $0$, and with 
appropriate time step parameter we have $\frac 1n W_2^2 ( x_k , \mu_R ) \leq \epsilon$. 
Since $\frac 1n W_2^2 (\mu_R , \mu )$
is negligible this implies $\frac 1n W_2^2 (x_k,\mu)\leq 2\epsilon$.
\end{proof}

\end{document}